\documentclass[11pt]{amsart}

\usepackage{amssymb}
\usepackage[colorlinks]{hyperref}
\usepackage[margin=1.25in]{geometry}
\usepackage{enumitem}
\usepackage{tikz}
\usetikzlibrary{matrix,arrows}
\usepackage{algorithm}
\usepackage{algpseudocode}

\theoremstyle{definition}
\newtheorem{theorem}{Theorem}[section]
\newtheorem{proposition}[theorem]{Proposition}

\newtheorem*{definition}{Definition}
\newtheorem{remark}[theorem]{Remark}

\newcommand{\im}{\operatorname{im}}

\author{Federico Galetto}
\title{Finite group characters on free resolutions}
\date{\today}

\address{Department of Mathematics and Statistics, Cleveland State University, Cleveland, OH, 44115-2215, USA}
\email{\href{mailto:f.galetto@csuohio.edu}{\nolinkurl{f.galetto@csuohio.edu}}}
\urladdr{\href{https://math.galetto.org}{\nolinkurl{https://math.galetto.org}}}



\begin{document}

\begin{abstract}
  Under reasonable assumptions, a group action on a module extends to
  the minimal free resolutions of the module.  Explicit descriptions
  of these actions can lead to a better understanding of free
  resolutions by providing, for example, convenient expressions for
  their differentials or alternative characterizations of their Betti
  numbers.  This article introduces an algorithm for computing
  characters of finite groups acting on minimal free resolutions of
  finitely generated graded modules over polynomial rings.
\end{abstract}

\maketitle

\section{Introduction}
\label{sec:introduction}

Let $\Bbbk$ be a field and let $R$ be a polynomial ring over $\Bbbk$
with a positive $\mathbb{Z}^r$-grading\footnote{This assumption
  includes many multigradings of interest in algebra and geometry. If
  preferred, the reader may safely assume to be working with a
  standard grading.} in the sense of \cite[Definition
4.2.4]{MR2159476}. The Betti numbers of a finitely generated graded
$R$-module $M$ are the integers
\begin{equation*}
  \beta_{i,j} (M) = \dim_\Bbbk \operatorname{Tor}^R_i (M, \Bbbk)_j,
\end{equation*}
for $i\in \mathbb{N}$ with $i\leqslant \dim R$ and
$j\in \mathbb{Z}^r$.  It is well understood that Betti numbers encode
many interesting algebraic and geometric properties of $M$ (see
\cite{MR1322960,MR2103875}).

Let $G$ be a finite group. Suppose there is a $\Bbbk$-linear action of
$G$ on both $R$ and $M$ that preserves degrees and is compatible with
the $R$-module structures. If $\Bbbk$ has characteristic zero or
positive characteristic not dividing the order of $G$ (the non modular
case), then $G$ is linearly reductive (see \cite[XVIII, Theorem
1.2]{MR1878556}). In this case, the $G$-action extends to a minimal
free resolution of $M$ in an essentially unique way (see \cite[\S
1]{MR3438710}). There has been significant interest in naturally
occurring examples of minimal free resolutions with such group actions
\cite{MR3215587,MR3779986,MR3799202,MR4043827,MR4105544,MR4108337,MR4155201,2010.06522,2012.13732,MR4136641}.

Understanding group actions on minimal free resolutions in terms of
representation theory can provide useful information about the
structure of the resolution or offer interesting combinatorial
interpretations of the Betti numbers. The goal is to describe the
$\Bbbk$-vector space $\operatorname{Tor}^R_i (M, \Bbbk)_j$ as a
representation of $G$. In the non modular case, this is completely
equivalent to describing the character of $G$ on
$\operatorname{Tor}^R_i (M, \Bbbk)_j$, i.e., the function $G\to \Bbbk$
that returns the trace of the action of an element $g\in G$ on
$\operatorname{Tor}^R_i (M, \Bbbk)_j$ (see \cite[Ch.\
2]{MR0450380}). We propose the following:
\begin{definition}
  The \emph{$(i,j)$-th Betti character of $G$ on $M$}, denoted
  $\beta^G_{i,j} (M)$, is the character of $G$ on
  $\operatorname{Tor}^R_i (M, \Bbbk)_j$.
\end{definition}
Note that the Betti character $\beta^G_{i,j} (M)$ evaluated at the
identity element of $G$ is the usual Betti number $\beta_{i,j} (M)$.

The goal of this work is to describe an algorithm (Algorithm
\ref{bettichar}) to compute Betti characters of finite groups in the
non modular case.  This is similar in spirit to our earlier work
\cite{MR3424030,MR3438710}, which dealt with the case of resolutions
with an action of a semisimple Lie group in characteristic zero.  Our
algorithm requires access to the following data:
\begin{itemize}
\item a minimal free resolution of $M$;
\item matrices for the action of $G$ on the generators of $R$;
\item matrices for the action of $G$ on the generators of $M$.
\end{itemize}
While free resolutions are generally quite difficult to construct, the
assumption is that a minimal free resolution of $M$ can be computed
with the aid of software.  Since characters are class functions (i.e.,
constant on conjugacy classes), it is enough to provide matrices for
the action of a set of representatives of the conjugacy classes of $G$
on $R$ and $M$. Since $G$ typically acts trivially on the
multiplicative identity of $R$, when $M$ is a quotient of $R$ by an
ideal, the action on the generator of $M$ is trivial. Therefore, in
most scenarios of interest, the requirements for running the algorithm
are minimal. The algorithm uses the action matrices to construct a new
minimal free resolution of $M$, then produces a map of complexes
between the given resolution and the new one (see Proposition
\ref{pro:1}). The trace of the individual maps in the map of complexes
are the entries of the desired Betti characters (see Theorem
\ref{thm:1}). While most currently available results describing finite
group actions on resolutions are limited to certain classes of modules
(e.g., monomial ideals) or certain groups (mainly symmetric groups),
the proposed algorithm bears no such restriction; as such, the author
hopes the algorithm will prove useful in exploring further examples of
finite group actions on free resolutions.

This paper is organized as follows: Section \ref{sec:theory}
contains the theoretical background, Section \ref{sec:algorithm}
describes the algorithm, and Section \ref{sec:examples} illustrates
the results with a few examples.

\section{Theory}
\label{sec:theory}

With the same notation as in Section \ref{sec:introduction}, let
$(F_\bullet, d_\bullet)$ be a minimal free resolution of $M$ as an
$R$-module and recall that the differentials
$d_i \colon F_i \to F_{i-1}$ are $G$-equivariant (i.e., they commute
with the action of the group; see \cite[Proposition 2.4.9]{MR3424030}
or \cite[\S~1]{MR3438710}). Computing $\operatorname{Tor}$ along this
resolution gives
\begin{equation*}
  \operatorname{Tor}^R_i (M, \Bbbk)_j \cong (F_i \otimes_R \Bbbk)_j,
\end{equation*}
where $(F_i \otimes_R \Bbbk)_j$ is the graded component of degree $j$
of the graded $\Bbbk$-vector space $F_i \otimes_R \Bbbk$.  For
$g\in G$, denote $\phi^g_i$ the $\Bbbk$-linear map which operates as
multiplication by $g$ on $F_i \otimes_R \Bbbk$.  The value of the
Betti character $\beta^G_{i,j} (M)$ at $g$ is the trace of $\phi^g_i$
restricted to degree $j$ (in both domain and codomain).  Note that
multiplication by $g$ on the modules $F_i$ is not $R$-linear, which
makes it inconvenient to work with directly.  We proceed to repackage
the same information using maps of complexes.

For each $i\geqslant 0$, fix an $R$-basis $\mathcal{E}_i$ of
$F_i$. Observe that the set $\{e\otimes 1 \,|\, e\in \mathcal{E}_i\}$
is a $\Bbbk$-basis of $F_i \otimes_R \Bbbk$.  For $g\in G$, let
$\psi^g_i \colon F_i \to F_i$ be the map of $R$-modules defined by
$\psi^g_i (e) = ge$ for all $e\in \mathcal{E}_i$, and extended by
$R$-linearity. Note that $\psi^g_i$ is invertible and
$(\psi^g_i)^{-1} = \psi^{g^{-1}}_i$.  For every $e\in \mathcal{E}_i$,
we have $\psi^g_i (e) = ge$ and
$\phi^g_i (e\otimes 1) = (ge) \otimes 1$, so $\psi^g_i$ and $\phi^g_i$
output $\Bbbk$-linear combinations with the same coefficients relative
to the bases $\mathcal{E}_i$ and
$\{e \otimes 1 \,|\, e\in \mathcal{E}_i\}$ respectively. Now the value
of the Betti character $\beta^G_{i,j} (M)$ at $g$ is the trace of
$\psi^g_i$ restricted to degree $j$, and $\psi^g_i$ is an $R$-linear
map.

In addition to being $R$-linear, the maps $\psi^g_i$ can be collected
into a map of complexes of $R$-modules.  For this purpose, we define
for every element $g\in G$ a map of $R$-modules
$d^g_i \colon F_i \to F_{i-1}$ by setting
$d^g_i = \psi^g_{i-1} d_i (\psi^g_i)^{-1}$.
\begin{proposition}\label{pro:1}
  The following are true for every $g\in G$.
  \begin{enumerate}[label=(\arabic*)]
  \item The collection $(F_\bullet, d^g_\bullet)$ is a minimal free
    resolution of $M$.
  \item The collection
    $\psi^g_\bullet \colon (F_\bullet, d_\bullet) \to (F_\bullet,
    d^g_\bullet)$ is a map of complexes.
  \item\label{item:2} If $[a_{uv}]$ is the matrix of $d_i$ relative to
    the bases $\mathcal{E}_i$ and $\mathcal{E}_{i-1}$, then the matrix
    of $d^g_i$ relative to the same bases is $[g^{-1} a_{uv}]$.
  \end{enumerate}
\end{proposition}
\begin{proof}
  Fix an arbitrary element $g\in G$ throughout the proof.
  \begin{enumerate}[wide,label=(\arabic*)]
  \item\label{item:1} The definition of $d^g_i$ immediately implies that the diagrams
    \begin{center}
      \begin{tikzpicture}[description/.style={fill=white,inner sep=2pt}]
        \matrix (m) [matrix of math nodes, row sep=3em,
        column sep=2.5em, text height=1.5ex, text depth=0.25ex]
        { F_{i-1} & F_i \\
          F_{i-1} & F_i \\ };
        \path[->,font=\scriptsize]
        (m-2-2) edge node[right] {$ \psi^g_i $} (m-1-2);
        \path[->,font=\scriptsize]
        (m-1-2) edge node[above] {$ d^g_i $} (m-1-1);
        \path[->,font=\scriptsize]
        (m-2-1) edge node[left] {$ \psi^g_{i-1} $} (m-1-1);
        \path[->,font=\scriptsize]
        (m-2-2) edge node[below] {$ d_i $} (m-2-1);
      \end{tikzpicture}
      \qquad
      \begin{tikzpicture}[description/.style={fill=white,inner sep=2pt}]
        \matrix (m) [matrix of math nodes, row sep=3em,
        column sep=2.5em, text height=1.5ex, text depth=0.25ex]
        { F_{i-1} & F_i \\
          F_{i-1} & F_i \\ };
        \path[->,font=\scriptsize]
        (m-1-2) edge node[right] {$ (\psi^g_i)^{-1} $} (m-2-2);
        \path[->,font=\scriptsize]
        (m-1-2) edge node[above] {$ d^g_i $} (m-1-1);
        \path[->,font=\scriptsize]
        (m-1-1) edge node[left] {$ (\psi^g_{i-1})^{-1} $} (m-2-1);
        \path[->,font=\scriptsize]
        (m-2-2) edge node[below] {$ d_i $} (m-2-1);
      \end{tikzpicture}
    \end{center}
    commute. Commutativity of the left diagram implies
    $\psi^g_i (\ker d_i) \subseteq \ker d^g_i$. Commutativity of the
    right diagram implies
    $(\psi^g_i)^{-1} (\ker d^g_i) \subseteq \ker d_i$, hence
    $\ker d^g_i \subseteq \psi^g_i (\ker d_i)$. Hence we have the
    equality $\psi^g_i (\ker d_i) = \ker d^g_i$. A similar reasoning
    shows $\psi^g_{i-1} (\im d_i) = \im d^g_i$. Since
    $(F_\bullet, d_\bullet)$ is a free resolution, there is an
    equality $\ker d_i = \im d_{i+1}$ for every integer
    $i\geqslant 1$. Applying the morphism $\psi^g_i$ and combining
    with the previous equalities, we deduce
    $\ker d^g_i = \im d^g_{i+1}$, thus showing
    $(F_\bullet, d^g_\bullet)$ is also a free resolution.

    Let $d_0 \colon F_0 \to M$ be the surjection that augments
    $(F_\bullet, d_\bullet)$ to a free resolution of the $R$-module
    $M$, so that $\ker d_0 = \im d_1$. Define the function
    $d^g_0 \colon F_0 \to M$ by setting $d^g_0 = d_0 (\psi^g_0)^{-1}$,
    so that $d^g_0 \psi^g_0 = d_0$. It follows immediately from the
    definition that $d^g_0$ is an epimorphism of
    $R$-modules. Moreover, the same kind of argument used in the
    previous paragraph shows that $\psi^g_0 (\ker d_0) = \ker
    d^g_0$. Combining equalities as before, we conclude that $d^g_0$
    augments $(F_\bullet, d^g_\bullet)$ to a free resolution of $M$.

    It remains to prove that $(F_\bullet, d^g_\bullet)$ is minimal,
    i.e., for every integer $i\geqslant 1$,
    $\im d^g_i \subseteq \mathfrak{m} F_{i-1}$ where $\mathfrak{m}$
    denotes the maximal ideal generated by the variables in $R$. We
    have previously established that
    $\psi^g_{i-1} (\im d_i) = \im d^g_i$, and we know
    $\im d_i \subseteq \mathfrak{m} F_{i-1}$ because
    $(F_\bullet, d_\bullet)$ is minimal. Combining this with the fact
    that $\psi^g_{i-1}$ is an isomorphism of $R$-modules, we have
    \begin{equation*}
      \im d^g_i = \psi^g_{i-1} (\im d_i)
      \subseteq \psi^g_{i-1} (\mathfrak{m} F_{i-1})
      = \mathfrak{m} \psi^g_{i-1} (F_{i-1})
      = \mathfrak{m} F_{i-1},
    \end{equation*}
    as desired.
  \item This statement is the commutativity of the left diagram in
    part \ref{item:1} of the proof.
  \item Let $\mathcal{E}_i = \{f_1,\dots,f_n\}$ and
    $\mathcal{E}_{i-1} = \{e_1,\dots,e_m\}$.  The entry in row $u$,
    column $v$ in the matrix of $d^g_i$ relative to $\mathcal{E}_i$
    and $\mathcal{E}_{i-1}$ is the coefficient of $e_u$ in the
    expression of $d^g_i (f_v)$ as a linear combination of
    $e_1,\dots,e_m$.  Since $d_i$ is $G$-equivariant, we have
    \begin{equation*}
      d^g_i (f_v)
      = \psi^g_{i-1} d_i (\psi^g_i)^{-1} (f_v)
      = \psi^g_{i-1} d_i (g^{-1} f_v)
      = \psi^g_{i-1} (g^{-1} d_i (f_v)).
    \end{equation*}
    If the matrix of $d_i$ relative to $\mathcal{E}_i$ and
    $\mathcal{E}_{i-1}$ is $[a_{uv}]$, then we get
    \begin{equation*}
      d^g_i (f_v)
      =  \psi^g_{i-1} \left( g^{-1} \sum_{i=1}^m a_{uv} e_u \right)
      =  \psi^g_{i-1} \left( \sum_{i=1}^m (g^{-1} a_{uv}) (g^{-1} e_u) \right),
    \end{equation*}
    because the $G$-action is compatible with the $R$-module
    operations. Finally, by definition of $\psi^g_{i-1}$, we
    obtain
    \begin{equation*}
      d^g_i (f_v)
      =  \sum_{i=1}^m (g^{-1} a_{uv}) \psi^g_{i-1} (g^{-1} e_u)
      =  \sum_{i=1}^m (g^{-1} a_{uv}) \psi^g_{i-1} \psi^{g^{-1}}_{i-1} (e_u)
      =  \sum_{i=1}^m (g^{-1} a_{uv}) e_u.
    \end{equation*}
    Therefore the coefficient of $e_u$ in the expression of
    $d^g_i (f_v)$ as a linear combination of $e_1,\dots,e_m$ is
    $g^{-1} a_{uv}$.
  \end{enumerate}
\end{proof}

We come to our main result, which, combined with Proposition
\ref{pro:1}, gives a practical way to compute $\beta^G_{i,j} (M) (g)$,
the value of the $(i,j)$-th Betti character of $G$ on $M$ at an
element $g\in G$. As observed earlier, $\beta^G_{i,j} (M) (g)$ is
equal to the trace of $\psi^g_i$ restricted to degree $j$.  The
problem is to determine matrices for the maps $\psi^g_i$.  As we will
show, this is not entirely necessary. In fact, it suffices to find
maps $\hat{\psi}^g_i$ that agree with $\psi^g_i$ in degree $j$.

Fix a free $R$-module $F_i$ with basis $\mathcal{E}_i$, and let
$F_{i,j}$ be the free direct summand of $F_i$ generated by the
elements of $\mathcal{E}_i$ of degree $j$. Let
$\pi_{i,j} \colon F_i \to F_{i,j}$ be the natural projection, which
does not depend on the chosen bases. Observe that $\pi_{i,j}$ is
$G$-equivariant because the action of $G$ sends an element of
$\mathcal{E}_i$ of degree $j$ to a $\Bbbk$-linear combination of
elements of $\mathcal{E}_i$ of degree $j$.

\begin{theorem}\label{thm:1}
  For every $g\in G$, if
  $\hat{\psi}^g_\bullet \colon (F_\bullet, d_\bullet) \to (F_\bullet,
  d^g_\bullet)$ is a map of complexes with
  $\hat{\psi}^g_0 = \psi^g_0$, then
  $\pi_{i,j} \hat{\psi}^g_i (e) = ge = \psi^g_i (e)$ for every
  $e\in \mathcal{E}_i$ of degree $j$.  In particular,
  $\beta^G_{i,j} (M) (g)$ is equal to the trace of $\hat{\psi}^g_i$
  restricted to degree $j$.
\end{theorem}

\begin{proof}
  By Proposition \ref{pro:1}, $(F_\bullet, d_\bullet)$ and
  $(F_\bullet, d^g_\bullet)$ are both minimal free resolutions of
  $M$. Moreover, $\hat{\psi}^g_\bullet$ and $\psi^g_\bullet$ are both
  maps of complexes lifting the $R$-automorphism of $M$ induced by
  $\hat{\psi}^g_0 = \psi^g_0$. By the Comparison Theorem (see
  \cite[Theorem 2.2.6]{MR1269324}), $\hat{\psi}^g_\bullet$ and
  $\psi^g_\bullet$ are homotopic, and therefore they induce the same
  map in homology (see \cite[Lemma 1.4.5]{MR1269324}). It follows
  that, for every $i\geqslant 0$,
  $\operatorname{im} (\hat{\psi}^g_i -\psi^g_i) \subseteq \mathfrak{m}
  F_i$.

  Consider an element $e\in \mathcal{E}_i$ of degree $j$.  We have
  $\hat{\psi}^g_i (e) = \psi^g_i (e) + f = ge+f$ for some
  $f\in \mathfrak{m} F_i$. Note that both $\hat{\psi}^g_i (e)$ and
  $ge$ have degree $j$ because $\hat{\psi}^g_i$ and the group action
  preserve degrees. It follows that $f$ also has degree $j$. Now apply
  the projection $\pi_{i,j}$ to obtain
  $\pi_{i,j} \hat{\psi}^g_i (e) = ge+\pi_{i,j} (f)$. Since
  $f\in \mathfrak{m} F_i$, we have that
  $\pi_{i,j} (f)\in \mathfrak{m} F_{i,j}$. However, since $F_{i,j}$ is
  a free module generated in degree $j$, the only element of degree
  $j$ in $\mathfrak{m} F_{i,j}$ is zero; hence, $\pi_{i,j} (f) = 0$.
  We conclude that $\pi_{i,j} \hat{\psi}^g_i (e) = ge = \psi^g_i (e)$,
  as desired. The equality between $\beta^G_{i,j} (M)$ and the trace
  of $\hat{\psi}^g_i$ restricted to degree $j$ is an immediate
  consequence.
\end{proof}

\begin{remark}
  After evaluation at representatives of all conjugacy classes of $G$,
  the character $\beta^G_{i,j} (M)$ can be decomposed into irreducible
  characters using the character table of $G$ (see for example
  \cite[Ch. 2]{MR0450380}). This allows further insight into the
  structure of $\operatorname{Tor}^R_i (M, \Bbbk)_j$ as a
  representation of $G$.
\end{remark}

\section{Examples}
\label{sec:examples}

Let $R = \mathbb{Q}[x_1,x_2,x_3,x_4]$ and consider the monomial ideal
\begin{equation*}
  I = \langle
  {x}_{1}{x}_{2},{x}_{1}{x}_{3},{x}_{2}{x}_{3},{x}_{1}{x}_{4},{x}_{2}{x}_{4},{x}_{3}{x}_{4}
  \rangle.
\end{equation*}
We regard $I$ as an $R$-module and compute a minimal free resolution
of $I$ with the help of the software Macaulay2 \cite{M2}. The
resolution $(F_\bullet, d_\bullet)$ has the following form:
\begin{center}
  \begin{tikzpicture}[description/.style={fill=white,inner sep=2pt}]
    \matrix (m) [matrix of math nodes, row sep=3em,
    column sep=2.5em, text height=1.5ex, text depth=0.25ex]
    { R(-2)^6 & R(-3)^8 & R(-4)^3 \\ };
    \path[->,font=\scriptsize]
    (m-1-2) edge node[above] {$ d_1 $} (m-1-1);
    \path[->,font=\scriptsize]
    (m-1-3) edge node[above] {$ d_2 $} (m-1-2);
  \end{tikzpicture}
\end{center}
with
\begin{equation*}
  d_1 =
  \begin{bmatrix}
    -{x}_{3}&0&-{x}_{4}&0&0&0&0&0\\
    {x}_{2}&-{x}_{2}&0&0&-{x}_{4}&0&0&0\\
    0&{x}_{1}&0&0&0&0&-{x}_{4}&0\\
    0&0&{x}_{2}&-{x}_{2}&{x}_{3}&-{x}_{3}&0&0\\
    0&0&0&{x}_{1}&0&0&{x}_{3}&-{x}_{3}\\
    0&0&0&0&0&{x}_{1}&0&{x}_{2}
  \end{bmatrix}
\end{equation*}
and
\begin{equation*}
  d_2 =
  \begin{bmatrix}
    {x}_{4}&{x}_{4}&0\\
    0&{x}_{4}&0\\
    -{x}_{3}&-{x}_{3}&0\\
    0&-{x}_{3}&{x}_{3}\\
    {x}_{2}&0&0\\
    0&0&-{x}_{2}\\
    0&{x}_{1}&0\\
    0&0&{x}_{1}
  \end{bmatrix}.
\end{equation*}
The symmetric group $\mathfrak{S}_4$ acts on $R$ by permuting
variables, i.e., $\sigma x_i = x_{\sigma (i)}$ for every
$\sigma \in \mathfrak{S}_4$. The ideal $I$ is stable under this
action, hence $\mathfrak{S}_4$ acts on the minimal free resolution
above. We are going to determine the Betti characters
$\beta^{\mathfrak{S}_4}_{i,j} (I)$.

Recall that conjugacy classes in $\mathfrak{S}_4$ are determined by cycle type \cite[\S~1.1]{MR1824028}, and therefore correspond to partitions of 4. We choose the following permutations to represent the conjugacy classes in $\mathfrak{S}_4$: in cycle notation, $(1234)$, $(123)$, $(12)(34)$, $(12)$, and $\operatorname{id}_{\mathfrak{S}_4}$. It is easy to describe the action of $G$ on the $\mathbb{Q}$-vector space spanned by the generators of $I$ as this is a permutation representation. This action determines the matrices of the maps $\psi^\sigma_0 \colon R(-2)^6 \to R(-2)^6$ for $\sigma \in \mathfrak{S}_4$. We have:
\begin{align*}
    \psi^{(1234)}_0 &=
    \begin{bmatrix}
      0&0&0&1&0&0\\
      0&0&0&0&1&0\\
      1&0&0&0&0&0\\
      0&0&0&0&0&1\\
      0&1&0&0&0&0\\
      0&0&1&0&0&0
    \end{bmatrix},&
    \psi^{(123)}_0 &=
    \begin{bmatrix}
      0&1&0&0&0&0\\
      0&0&1&0&0&0\\
      1&0&0&0&0&0\\
      0&0&0&0&0&1\\
      0&0&0&1&0&0\\
      0&0&0&0&1&0
    \end{bmatrix},\\
    \psi^{(12)(34)}_0 &=
    \begin{bmatrix}
      1&0&0&0&0&0\\
      0&0&0&0&1&0\\
      0&0&0&1&0&0\\
      0&0&1&0&0&0\\
      0&1&0&0&0&0\\
      0&0&0&0&0&1
    \end{bmatrix},&
    \psi^{(12)}_0 &=
    \begin{bmatrix}
      1&0&0&0&0&0\\
      0&0&1&0&0&0\\
      0&1&0&0&0&0\\
      0&0&0&0&1&0\\
      0&0&0&1&0&0\\
      0&0&0&0&0&1
    \end{bmatrix}.
\end{align*}
From the trace of these matrices, we can determine the Betti character
\begingroup
\renewcommand{\arraystretch}{1.6}
\begin{equation*}
  \begin{array}{c|ccccc}
    & (1234) & (123) & (12)(34) & (12) & \operatorname{id}_{\mathfrak{S}_4}\\ \hline
    \beta^{\mathfrak{S}_4}_{0,2} (I) &0 &0 &2 &2 &6
  \end{array}
\end{equation*}
\endgroup
where the last value is the trace of the identity map on the module
$F_0$, which is the same as the rank of $F_0$.

We illustrate the computation of the other Betti characters by
focusing on the element $(1234)$. The resolution
$(F_\bullet, d^{(1234)}_\bullet)$ has the following form:
\begin{center}
  \begin{tikzpicture}[description/.style={fill=white,inner sep=2pt}]
    \matrix (m) [matrix of math nodes, row sep=3em,
    column sep=3.5em, text height=1.5ex, text depth=0.25ex]
    { R(-2)^6 & R(-3)^8 & R(-4)^3 \\ };
    \path[->,font=\scriptsize]
    (m-1-2) edge node[above] {$ d^{(1234)}_1 $} (m-1-1);
    \path[->,font=\scriptsize]
    (m-1-3) edge node[above] {$ d^{(1234)}_2 $} (m-1-2);
  \end{tikzpicture}
\end{center}
with
\begin{equation*}
  d^{(1234)}_1 =
  \begin{bmatrix}
    -{x}_{2}&0&-{x}_{3}&0&0&0&0&0\\
    {x}_{1}&-{x}_{1}&0&0&-{x}_{3}&0&0&0\\
    0&{x}_{4}&0&0&0&0&-{x}_{3}&0\\
    0&0&{x}_{1}&-{x}_{1}&{x}_{2}&-{x}_{2}&0&0\\
    0&0&0&{x}_{4}&0&0&{x}_{2}&-{x}_{2}\\
    0&0&0&0&0&{x}_{4}&0&{x}_{1}
  \end{bmatrix}
\end{equation*}
and
\begin{equation*}
  d^{(1234)}_2 =
  \begin{bmatrix}
    {x}_{3}&{x}_{3}&0\\
    0&{x}_{3}&0\\
    -{x}_{2}&-{x}_{2}&0\\
    0&-{x}_{2}&{x}_{2}\\
    {x}_{1}&0&0\\
    0&0&-{x}_{1}\\
    0&{x}_{4}&0\\
    0&0&{x}_{4}
  \end{bmatrix}.
\end{equation*}
By Proposition \ref{pro:1} \ref{item:2}, the matrices of the
differentials $d^{(1234)}_i$ are obtained by applying the permutation
$(1234)^{-1} = (1432)$ to the matrices of the differentials $d_i$.

Next we compute $\hat{\psi}^{(1234)}_1$ by solving the equation $\psi^{(1234)}_0 d_1 = d^{(1234)}_1 \hat{\psi}^{(1234)}_1$. In this example, the solution can be found by direct inspection although in general it would require solving a system of linear equations. The result is
\begin{equation*}
  \hat{\psi}^{(1234)}_1 =
  \begin{bmatrix}
    0&0&-1&1&0&0&0&0\\
    0&0&-1&0&0&0&0&0\\
    0&0&0&0&-1&1&0&0\\
    0&0&0&0&-1&0&0&0\\
    0&0&0&0&0&0&-1&1\\
    0&0&0&0&0&0&-1&0\\
    1&0&0&0&0&0&0&0\\
    0&1&0&0&0&0&0&0
  \end{bmatrix}.
\end{equation*}
Note that this is an invertible matrix with entries in $\mathbb{Q}$ whose domain and codomain are both generated in degree 3. It follows that this is actually the matrix of $\psi^{(1234)}_1$, the map defined to act like $(1234)$ on the given basis of $F_1$.

Similarly, we compute $\hat{\psi}^{(1234)}_2$ by solving the equation $\hat{\psi}^{(1234)}_1 d_2 = d^{(1234)}_2 \hat{\psi}^{(1234)}_2$. The result is
\begin{equation*}
  \hat{\psi}^{(1234)}_2 =
  \begin{bmatrix}
    0&-1&1\\
    1&1&0\\
    0&1&0
  \end{bmatrix}.
\end{equation*}

Repeating the process for the representatives of the other conjugacy
classes, we find a complete description of the Betti characters, which
we summarize in the table below.
\begingroup
\renewcommand{\arraystretch}{1.6}
\begin{equation*}
  \begin{array}{c|ccccc}
    & (1234) & (123) & (12)(34) & (12) & \operatorname{id}_{\mathfrak{S}_4}\\ \hline
    \beta^{\mathfrak{S}_4}_{0,2} (I) &0 &0 &2 &2 &6\\
    \beta^{\mathfrak{S}_4}_{1,3} (I) &0 &-1 &0 &0 &8\\
    \beta^{\mathfrak{S}_4}_{2,4} (I) &1 &0 &-1 &-1 &3\\
  \end{array}
\end{equation*}
\endgroup The reader may verify the Betti characters found here are
the same as the characters of the representations found in
\cite[\S~4]{MR4105544}. We also mention that $I$ is a symmetric
shifted ideal and its equivariant resolution can be found as described
in \cite[Theorem 6.2, Example 6.3]{MR4108337}.

For completeness, we include matrices of the maps
$\hat{\psi}^\sigma_\bullet$ for the representatives of the other
conjugacy classes (except the identity). These matrices were found
using our implementation of Algorithm \ref{bettichar} in Macaulay2.
\begin{align*}
    \hat{\psi}^{(123)}_1 &=
    \left[
      \begin{smallmatrix}
        -1&1&0&0&0&0&0&0\\
        -1&0&0&0&0&0&0&0\\
        0&0&0&0&1&0&0&0\\
        0&0&0&0&1&-1&0&0\\
        0&0&0&0&0&0&1&0\\
        0&0&0&0&0&0&1&-1\\
        0&0&1&0&0&0&0&0\\
        0&0&0&1&0&0&0&0
      \end{smallmatrix}
    \right],&
    \hat{\psi}^{(123)}_2 &=
    \left[
      \begin{smallmatrix}
        0&1&0\\
        -1&-1&0\\
        0&-1&1
      \end{smallmatrix}
    \right],\\
    \hat{\psi}^{(12)(34)}_1 &=
    \left[
      \begin{smallmatrix}
        0&0&1&0&0&0&0&0\\
        0&0&1&-1&0&0&0&0\\
        1&0&0&0&0&0&0&0\\
        1&-1&0&0&0&0&0&0\\
        0&0&0&0&0&0&-1&1\\
        0&0&0&0&0&0&0&1\\
        0&0&0&0&-1&1&0&0\\
        0&0&0&0&0&1&0&0
      \end{smallmatrix}
    \right],&
    \hat{\psi}^{(12)(34)}_2 &=
    \left[
      \begin{smallmatrix}
        0&-1&1\\
        -1&0&-1\\
        0&0&-1
      \end{smallmatrix}
    \right],\\
    \hat{\psi}^{(12)}_1 &=
    \left[
      \begin{smallmatrix}
        1&0&0&0&0&0&0&0\\
        1&-1&0&0&0&0&0&0\\
        0&0&1&0&0&0&0&0\\
        0&0&1&-1&0&0&0&0\\
        0&0&0&0&0&0&1&0\\
        0&0&0&0&0&0&0&1\\
        0&0&0&0&1&0&0&0\\
        0&0&0&0&0&1&0&0
      \end{smallmatrix}
    \right],&
    \hat{\psi}^{(12)}_2 &=
    \left[
      \begin{smallmatrix}
        0&1&0\\
        1&0&0\\
        0&0&-1
      \end{smallmatrix}
    \right].
\end{align*}

\section{Algorithm}
\label{sec:algorithm}

We generalize the ideas of the example from Section \ref{sec:examples}
and formalize them into Algorithm \ref{bettichar}, which we present in
pseudocode below. We refer again to the notations introduced in
Sections \ref{sec:introduction} and \ref{sec:theory}.

We start by describing some auxiliary functions. The implementation
will depend on the system but should be fairly straightforward.

\begin{description}[wide,font=\normalfont]
\item[\textsc{degrees}($A$)] Takes as input the matrix $A$ of a map
  between graded free $R$-modules. It returns a list of degrees
  $\{\deg (e_1),\dots,\deg (e_n)\}$, where $\{e_1,\dots,e_n\}$ is the
  basis used for the domain of $A$.
\item[\textsc{trace}($A$,$j$)] Takes as input the matrix $A$ of a map
  between two copies of the same graded free $R$-module and a degree
  $j$. It returns the trace of the submatrix of $A$ formed by taking
  only rows and columns corresponding to basis elements of degree $j$.
\item[\textsc{action}($g$,$A$)] Takes as input an element $g\in G$ and
  the matrix $A=[a_{uv}]$ of a map between graded free $R$-modules. It
  returns the matrix $[g a_{uv}]$ obtained by acting with $g$ on each
  entry of $A$.
\item[\textsc{factor}($A$,$B$)] Takes as input two matrices $A$ and
  $B$ corresponding to maps between graded free $R$-modules. The maps
  should have the same codomain. Also $A$ and $B$ should be expressed
  relative to the same basis of their common codomain. It returns a
  matrix $C$ with homogeneous entries in $R$ such that $A = BC$.
\end{description}

\begin{remark}
  In Macaulay2, the function \textsc{action} can be implemented as a
  substitution that sends each variable $x_i$ in the ring
  $R = \Bbbk [x_1,\dots, x_n]$ to the element $g x_i$, which can be
  expressed as a linear combination of the variables. Similarly, the
  function \textsc{factor} is already implemented in Macaulay2 via the
  operator \texttt{//}. We also point out that Macaulay2 has the
  ability to compute extensions of module maps to maps of complexes via
  the function \texttt{extend}; this can be used to obtain the map of
  complexes $\hat{\psi}^g_\bullet$ starting from a matrix of the map
  $\psi^g_0$ for a given $g\in G$.
\end{remark}

\begin{remark}
  As observed in Section \ref{sec:examples}, the use of
  \textsc{factor} in the context of Algorithm \ref{bettichar} requires
  solving a system of linear equations. In systems where a matrix
  factorization function is implemented (such as Macaulay2), this
  could rely on Gr\"obner bases calculations. Moreover, the result of
  such a matrix factorization is, in principle, not unique. However,
  Theorem \ref{thm:1} guarantees any map of complexes
  $\hat{\psi}^g_\bullet$ that extends $\psi^g_0$ can be used to
  determine the value of the Betti character $\beta^G_{i,j} (M)$ at
  $g$. Therefore how factorization is actually implemented should not
  be a source of concern.
\end{remark}

Given a set $\{g_1,\dots,g_r\}$ of representatives of the conjugacy
classes of $G$, the character $\chi$ of a finite dimensional
representation of $G$ can be encoded into the list
$\{\chi(g_1),\dots, \chi(g_r)\}$. This data type will be used by our
algorithm to return Betti characters. Moreover, the collection of all
nonzero Betti characters can be assembled into a hash table with keys
$(i,j)$ and corresponding values $\beta^G_{i,j} (M)$, where $i$ is a
homological degree and $j$ is a degree.

\begin{algorithm}[hb]
  \caption{Betti characters}\label{bettichar}
  \algrenewcommand\algorithmicrequire{\textbf{Input:}}
  \algrenewcommand\algorithmicensure{\textbf{Output:}}
  \begin{algorithmic}[1]
    \Require{
      \begin{itemize}
      \item $D=\{D_1,\dots,D_n\}$, a list with $D_i$ the matrix
        of the differential $d_i$ relative to the bases $\mathcal{E}_i$
        and $\mathcal{E}_{i-1}$
      \item $P=\{P_1,\dots,P_r\}$, a list with $P_k$ the matrix of
        $\psi^{g_k}_0$ for a set $\{g_1,\dots,g_r\}$ of representatives
        of the conjugacy classes of $G$
      \end{itemize}
    }
    \Ensure{$B$, a hash table with keys $(i,j)$ and values $B_{i,j} = \beta^G_{i,j} (M)$ the nonzero Betti characters of $G$ on $M$ represented as lists}
    \Statex
    \Function{BettiCharacters}{$D,P$}
    \State $B \gets \{\}$
    \Comment{initialize empty hash table}
    \For{$k\gets 1,\dots,r$}
    \Comment{loop for homological degree $0$}
    \State $Q_k \gets P_k$
    \Comment{matrix of $\psi^{g_k}_0$}
    \ForAll{$j\in$ \Call{Degrees}{$Q_k$}}
    \State $B_{0,j} \gets B_{0,j} \cup \{$\Call{Trace}{$Q_k$,$j$}$\}$
    \Comment{append to (possibly empty) list}
    \EndFor
    \EndFor
    \For{$i\gets 1,\dots,n$}
    \Comment{loop for homological degrees $>0$}
    \For{$k\gets 1,\dots,r$}
    \State $C_k \gets $ \Call{Action}{$g_k^{-1}$,$D_i$}\label{line:1}
    \Comment{matrix of $d^{g_k}_i$}
    \State $Q_k \gets $ \Call{Factor}{$Q_k D_i$,$C_k$}\label{line:2}
    \Comment{matrix of $\hat{\psi}^{g_k}_i$}
    \ForAll{$j\in$ \Call{Degrees}{$Q_k$}}
    \State $B_{i,j} \gets B_{i,j} \cup \{$\Call{Trace}{$Q_k$,$j$}$\}$\label{line:3}
    \Comment{append to (possibly empty) list}
    \EndFor
    \EndFor
    \EndFor
    \State \textbf{return} $B$
    \EndFunction
  \end{algorithmic}
\end{algorithm}

The computation of $d^{g_k}_i$ in line \ref{line:1} follows
Proposition \ref{pro:1} \ref{item:2}. The construction of
$\hat{\psi}^{g_k}_i$ via a matrix factorization in line \ref{line:2}
ensures that $\hat{\psi}^{g_k}_\bullet$ is a map of complexes
extending $\psi^{g_k}_0$.  The computation of the values of the Betti
characters in line \ref{line:3} is justified by Theorem \ref{thm:1}.

\begin{remark}
  Many interesting equivariant resolutions arise as minimal free
  resolutions of quotient rings $R/I$ where $I$ is a $G$-stable ideal,
  i.e., $GI\subseteq I$. This is the case of the example in Section
  \ref{sec:examples}. Note that a minimal free resolution
  $(F_\bullet, d_\bullet)$ of $M=R/I$ has $F_0 = R$ with a trivial
  action of $G$ on
  $\operatorname{Tor}^R_0 (M,\Bbbk) \cong F_0 \otimes_R \Bbbk \cong
  \Bbbk$. In this setting, the matrices $P_1,\dots,P_r$ in the input
  of Algorithm \ref{bettichar} are 1 by 1 identity matrices, hence
  they can be omitted from the input. This implies that the
  description of the group action on the resolution is entirely
  determined by group action on the ring, which is computed by the
  function \textsc{action}.
\end{remark}

\begin{remark}
  Algorithm \ref{bettichar} assumes the action of the group $G$ on the
  minimal free resolution $(F_\bullet, d_\bullet)$ can be explicitly
  described in homological degree zero by matrices $ \psi^g_0 $ for
  certain $g\in G$. The action in higher homological degrees is then
  computed by recursive matrix factorizations. More explicitly, if the
  matrix $ \psi^g_{i-1} $ is known, the matrix $ \psi^g_i $ can be
  computed by a matrix factorization via the following diagram.
  \begin{center}
    \begin{tikzpicture}[description/.style={fill=white,inner sep=2pt}]
      \matrix (m) [matrix of math nodes, row sep=3em,
      column sep=2.5em, text height=1.5ex, text depth=0.25ex]
      { F_{i-1} & F_i \\
        F_{i-1} & F_i \\ };
      \path[->,font=\scriptsize]
      (m-2-2) edge node[right] {$ \psi^g_i $} (m-1-2);
      \path[->,font=\scriptsize]
      (m-1-2) edge node[above] {$ d^g_i $} (m-1-1);
      \path[->,font=\scriptsize]
      (m-2-1) edge node[left] {$ \psi^g_{i-1} $} (m-1-1);
      \path[->,font=\scriptsize]
      (m-2-2) edge node[below] {$ d_i $} (m-2-1);
    \end{tikzpicture}
  \end{center}
  Note that $ \psi^g_{i-1} $ can also be computed from $ \psi^g_i $
  via a matrix factorization. Therefore it is possible to extend the
  algorithm so that the starting point for computing the action on a
  resolution can be given in arbitrary homological degree rather than
  degree zero.
\end{remark}

\begin{remark}
  In the setting of Section \ref{sec:introduction}, the graded
  components of $M$ are also finite dimensional representations of
  $G$, and the reader may be interested in finding their
  characters. This is a much easier task than computing Betti
  characters. For completeness, we briefly outline a strategy for
  computing characters of graded components.
  
  Consider a (minimal) presentation $d_1 \colon F_1 \to F_0$ of
  $M$. Fix a basis of $F_0$ and a module term ordering on the terms of
  $F_0$ relative to this basis. Let $j$ be a degree such that the
  graded component $M_j$ is nonzero, and let
  $\mathcal{B} = \{ b_1,\dots,b_s \}$ be the set of terms of degree
  $j$ in $F_0$ not contained in the leading term module of
  $\operatorname{im} d_1$. By Macaulay's Basis Theorem \cite[Theorem
  1.5.7]{MR1790326}, the residue classes $\bar{b}_1,\dots, \bar{b}_s$
  modulo $\operatorname{im} d_1$ form a $\Bbbk$-basis of $M_j$.  For
  $g\in G$, we may write
  \begin{equation*}
    g \bar{b}_v = \sum_{u=1}^s a_{uv} \bar{b}_u
  \end{equation*}
  for some $a_{uv}\in\Bbbk$. The character of $G$ on $M_j$ evaluated
  at $g$ is the trace of the matrix $[a_{uv}]$. It is worth noting
  that the input of Algorithm \ref{bettichar} provides all the
  information needed to compute these characters (and more).
\end{remark}


\newcommand{\etalchar}[1]{$^{#1}$}
\def\cprime{$'$} \def\Dbar{\leavevmode\lower.6ex\hbox to 0pt{\hskip-.23ex
  \accent"16\hss}D} \def\Dbar{\leavevmode\lower.6ex\hbox to 0pt{\hskip-.23ex
  \accent"16\hss}D} \def\Dbar{\leavevmode\lower.6ex\hbox to 0pt{\hskip-.23ex
  \accent"16\hss}D} \def\Dbar{\leavevmode\lower.6ex\hbox to 0pt{\hskip-.23ex
  \accent"16\hss}D}

\end{document}